\documentclass[twocolumn,journal]{IEEEtran}
\IEEEoverridecommandlockouts

\usepackage[applemac]{inputenc}
\usepackage{graphics} 
\usepackage{epsfig} 
\usepackage{epstopdf}
\usepackage{algorithm}
\usepackage[noend]{algpseudocode}
\usepackage{pgfplots} 

\usepackage{amsmath} 
\usepackage{amssymb}  
\usepackage{amsthm}
\usepackage{epstopdf}
\usepackage{graphicx,xcolor}
\usepackage{caption}
\usepackage{subcaption}
\usepackage{courier}
\usepackage{mathrsfs}
\usepackage{tikz}
\usetikzlibrary{calc,patterns,angles,quotes}

\makeatletter
\def\BState{\State\hskip-\ALG@thistlm}
\makeatother

\newtheorem{thm}{Theorem}

\newtheorem{lem}{Lemma}
\newtheorem{cor}{Corollary}

\title{\LARGE \bf
Controller Synthesis for Safety of \\Physically-Viable Data-Driven Models
}

\author{Mohamadreza Ahmadi, Arie Israel, and Ufuk Topcu
\thanks{M. Ahmadi and U. Topcu are with the Department of Aerospace Engineering and Engineering Mechanics, and the Institute for Computational Engineering and Sciences (ICES), University of Texas, Austin, 201 E 24th St, Austin, TX 78712. Arie Israel is with the Department of Mathematics, University of Texas at Austin, 2515 Speedway Stop C1200 
Austin, Texas 78712-1202, USA.  e-mail: (\{mrahmadi, arie,  utopcu\}@utexas.edu). The work has been supported partly by AFRL  FA8650-15-C-2546, AFRL  FA8650-16-C-2610, DARPA W911NF-16-1-0001 and ARO W911NF-15-1-0592.
}
}

\begin{document}

\maketitle
\thispagestyle{empty}
\pagestyle{empty}

\begin{abstract}

We consider the problem of designing finite-horizon safe controllers  for a dynamical system for which no explicit analytical model exists and  limited data only along a single trajectory  of the system are available. Given samples of the states and  inputs of the system,  and additional  side information in terms of  regularity of the evolution of the  states  and conservation laws, we synthesize a controller such that the evolution of the states avoid some  pre-specified unsafe set over a given finite horizon.  Motivated by recent results on Whitney's extension theorem, we use   piecewise-polynomial approximations of the trajectories based on the data along with the regularity side information to formulate a data-driven differential inclusion model that can predict the evolution of the trajectories. For these classes of data-driven differential inclusions, we propose a safety analysis  theorem based on barrier certificates. As a corollary of this theorem, we demonstrate that we can design controllers ensuring safety of the solutions to the data-driven differential inclusion over a finite horizon. From a computational standpoint, our results are cast into a set of sum-of-squares  programs whenever the certificates are parametrized by polynomials of fixed degree and the sets are semi-algebraic. This computational method allows incorporating side information in terms of  conservation laws and integral constraints using Positivstellensatz-type arguments.


\end{abstract}

\section{INTRODUCTION}


Learning-based methods have been successful in modeling and controlling many dynamical systems~\cite{6315769,Lenz15deepmpclearning}. These methods often require a large number of system runs (e.g., trajectories from different initial conditions) over a long time span,   to achieve reasonable performance. However, for a relatively broad class of systems, collecting large sums of data can be too cumbersome or not viable at all. The scarcity of the available data is particularly noticeable for safety-critical systems, in which an abrupt change in system model can result in catastrophic control failures. For instance, it is not practically possible to test and collect data from all possible  failure scenarios for an unmanned vehicle~\cite{f15Israel}. Furthermore, for safety-critical systems, we need to construct a model that can be used to predict the behavior of the system, and the construction of and control with such models should not incur a high computational cost (as opposed to conventional learning methods).

Recent studies have shown that certain mathematical models in the form of differential equations can be extracted from data~\cite{Schmidt81}. In particular, \cite{HGH15} studied the problem of finding system dynamics when the system follows Lagrangian mechanics. Also, see~\cite{tran2016exact} for a method that can extract chaotic polynomial differential equations from noisy data and relies on the ergodicity property of the data such that the central limit theorem can be applied. However, these methods often require large amounts of training data, which may not  be available.  In the control literature, system analysis  based on input-output data or input-state data is not new. System identification techniques~\cite{Ljung20101} have looked into the problem of finding a model of the system based on data. Nevertheless, the available methods  are ``data-hungry'' or computationally expensive, especially if they require a validation stage.  Adaptive control techniques~\cite{Tao20142737} also studied controller synthesis methods for systems in which the system model is known up to a parametrization. Such parametrization of the system dynamics is not often available,  for instance, in the case of an abrupt system change.

One fundamental issue for safety critical systems is to   ensure the system behaves \emph{safely} or  guarantee that the system avoids certain \emph{unsafe} behavior. If the system model is given, verifying safety is a familiar subject to the control community~\cite{1215682,Prajna2006117}. One of the methods for safety verification relies on the construction of a function of the states, called the \emph{barrier certificate}~\cite{Prajna2006117}. Barrier certificates have shown to be useful in several system analysis  and control problems inluding bounding moment functionals of stochastic systems~\cite{ahmadi2016optimization}, safety analysis of systems described by partial differential equations~\cite{AHMADI201733}, safety verification of refrigeration systems~\cite{7602538}, and control of a swarm of silk moths~\cite{4587085}. It was also proved in~\cite{7236867} that for every safe dynamical system (defined in the appropriate sense), there exists a barrier certificate. To the authors' knowledge, the only article that applied barrier certificates for system analysis  based on data is~\cite{7402508}. However, the latter method requires large amounts of data, as well. 

Apart from safety analysis, several studies considered the so called control barrier functions as a means to render the solutions of a system safe. In~\cite{WIELAND2007462}, the authors, inspired by the notion of control Lyapunov functions~\cite{sontag89}, introduced control barrier functions. This formulation, however, requires a one-dimensional control signal. In the same vein, \cite{ROMDLONY201639} demonstrated that one can simultaneously search for a safe and stabilizing controller. Alternatively, \cite{ames2016control} proposed control barrier functions with a fixed logarithmic structure as a function of the unsafe set. It was also shown that  this control barrier function structure allows the safe controller synthesis problem to be solved by a set of quadratic programs, satisfies robustness properties such as input-to-state stability with regards to the perturbations to the vector field, and leads to Lipschitz continuous control laws ~\cite{XU201554}. This result was extended in~\cite{7524935} to exponential control barrier functions based on techniques from linear control theory.


 
  In this paper, we study safety analysis  and safe control of systems for which limited data in terms of state and input samples  from a single trajectory   is available (by \emph{limited},  we imply that the number of data samples is not large enough for determining the complete dynamics using a system identification or machine learning method). This research is motivated by the recent works on Whitney's extension problem~\cite{Fefferman2013} on finding interpolants with optimal regularity constants in terms of the $\mathcal{C}^n$-norms. Following the footsteps~\cite{Fefferman2013}, \cite{herbert2014computing} showed that the cubic spline polynomials are the best interpolants in terms of minimizing the Lipschitz constant ($\mathcal{C}^0$-norm) for twice continuously differentiable trajectories and proposed computational methods on how to find these interpolants. Accordingly, we build a data-driven differential inclusion model that can be used to predict the evolution of system trajectories based on the piecewise-polynomial (cubic spline) approximation of the state  and input data and some regularity information on the evolution of system state.
  
   Equipped with this data-driven model in terms of convex differential inclusions,  we formulate a safety analysis  theorem based on barrier certificates for differential inclusions using notions from set-valued analysis~\cite{aubinsetvalued} and the theory of differential inclusions~\cite{aubincelina}. This barrier certificate is a possibly non-smooth function of the states and time satisfying two inequalities along the solutions of the data-driven differential inclusion. Then, we present conditions to synthesize controllers ensuring safety of the latter differential inclusions, without imposing any \textit{a priori} fixed structure on the barrier functions.  We evince that both the analysis  and controller synthesis methods can be cast into a set of sum-of-squares programs whenever the certificates are parametrized by polynomials of fixed degree and the sets are semi-algebraic. We further demonstrate that our formulation can accommodate more side information such as conservations laws in terms of algebraic inequalities/equalities and hard integral quadratic constraints~\cite{6915700} based on the application of Positivstellensatz. 
   
   Preliminary results on  this work were discussed in~\cite{AIT17}, in which we brought forward the safety analysis method based on barrier certificates for the data-driven models. The current paper in addition to providing more detailed proofs to the results in~\cite{AIT17} with more discussion of the literature, proposes  a controller synthesis algorithm for safety of the data-driven
models, formulates a method to include  side information in terms of algebraic and
integral inequalities, such hard IQCs, and includes  a computational method based on sum-of-squares optimization to
synthesize the safe controllers.


The paper is organized as follows. The next section presents the notation and  some preliminary mathematical definitions. In Section~\ref{sec:main}, we show how the data-driven differential inclusion models are constructed from the piece-wise polynomial approximation of the data. In Section~\ref{sec:controlsafe}, we propose a method based on barrier certificates for safety analysis  of differential inclusions and a method for designing safe controllers for systems with limited data. Section~\ref{sec:compute} describes a computational approach for finding barrier certificates and designing safe controllers based on polynomial optimization, and Section~\ref{sec:side} delineates an approach using Positivstellensatz to include side infromation.   In Section~\ref{sec:numresults}, we illustrate the proposed method by two examples. Finally, Section~\ref{sec:conclusions} concludes the paper and provides directions for future research.

\textbf{Notation:}
The notations employed in this paper are relatively straightforward. $\mathbb{R}_{\ge 0}$ denotes the set $[0,\infty)$. $\| \cdot \|$ denotes the Euclidean  norm on $\mathbb{R}^n$. $\mathcal{R}[x]$ accounts for the set of polynomial functions with real coefficients in $x \in \mathbb{R}^n$,  and $\Sigma \subset\mathcal{R}$ is the subset of polynomials with an sum of squares decomposition; i.e,  $p \in  \Sigma[x]$ if and only if there are $p_i \in \mathcal{R}[x],~i \in \{1, \ldots ,k\}$ such that $p = p_i^2 + \cdots +p_k^2$.  We denote by $\mathcal{C}^m(X)$, with $X \subseteq \mathbb{R}^n$, the space of $m$-times continuously differentiable functions and by  $\partial^m = \frac{\partial^m}{\partial x^m}$ the derivatives up to order $m$. For $f \in \mathcal{C}^m(X)$, we denote by $\| f \|_{\mathcal{C}^m}$ the $\mathcal{C}^m$-norm given by
$$
\| f \|_{\mathcal{C}^m} = \max_{\alpha \le m}\sup_{x \in X} | \partial^{\alpha} f(x) |.
$$
For $f \in \mathcal{C}^m(X)$ and $x \in X$, we denote by $J_x(f)$  the $m$th degree Taylor polynomial of $f$ at $x$
$$
J_x(f)(x') = \sum_{\alpha \le m} \frac{\partial^\alpha f(x) (x'-x)^\alpha}{\alpha !}.
$$
Note that $J_x(f) \in \mathcal{R}[x]$. We denote by $\mathcal{L}(X;Y)$ the set of Lebesgue measurable functions mapping $X$ to $Y$ and similarly $\mathcal{C}^n(X;Y)$ the set of $n$-times continuously differentiable functions mapping $X$ to $Y$, where $X$ and $Y$ are subsets of the Euclidean space. $2^{A}$ signifies the power set of $A$.  Finally, for a finite set $A$, we denote by $co\{A\}$ the convex hull of the set~$A$.
\section{Preliminaries}

In this section, we discuss some preliminary mathematical notions and results that will be employed in the sequel.

\subsection{Whitney's Extension Problem}

Whitney's extension problem is concerned with the question of whether, given data on a function $f$, i.e., $\{ \partial^m f_i \}_{i=1}^N$ corresponding to $\{x_i\}_{i=1}^N$ such that $\partial^m f_i = \partial^m f_i(x_i)$, one can find a $\mathcal{C}^m$-function that approximates $f$. It can be described as follows. Suppose we are given an arbitrary subset $D \subset \mathbb{R}^n$ and a
function $f : D \to \mathbb{R}$. How can we
determine whether there exists a function $F \in \mathcal{C}^m(\mathbb{R}^n)$ such that $F = f$ on $D$? 

Whitney indeed addressed this problem for the case $n=1$.

\begin{thm}[Whitney's Extension Theorem \cite{citeulike:12551299}]
Let $E \subset \mathbb{R}^n$ be a closed set, and let $\{{P}_x\}_{x \in {E}}$ be a family of polynomials ${P}_x \in \mathcal{R}[x]$ indexed by the points of $E$. Then the following are equivalent.
\begin{enumerate}
\item[A.] There exists $F \in \mathcal{C}^m(\mathbb{R}^n)$ such that $J_x(F) = P_x$ for each $x \in E$.
\item[B.] There exists a real number $M > 0$ such that
$$
| \partial^\alpha P_x(x) | \le M \quad \text{for} \quad | \alpha | \le m,~x \in E.
$$
\end{enumerate}
\end{thm}

Recently in \cite{10.2307/3597349} and \cite{fefferman2009} considered a more general problem. That is,  given  $\{  f_i=f(x_i) \}_{i=1}^N$, the problem of computing a function $F \in \mathcal{C}^m(\mathbb{R}^n)$ and a real number $M \ge 0$ such that
$$
\|F\|_{\mathcal{C}^m} \le M \quad \text{and} \quad |F(x) - f(x)| \le M\sigma(x),\quad \forall x \in E.
$$
The function $\sigma:\mathbb{R}^n \to \mathbb{R}_{\ge0}$ is determined by the problem under study and from "observations".  The function also serves as a ``tolerance''. It implies that the graph of $F$ passes sufficiently close to the $N$ given data points. 

Computing such a function $F$ in the general form is a cumbersome task and amounts to computing sets containing $F$~\cite{fefferman2011,10.2307/40345469}. In this paper, instead of considering general interpolants of data, we focus on piecewise-polynomial approximations for which construction algorithm are widely available~\cite{de2001practical}.

\subsection{Piecewise-Polynomial Approximation: B-Splines}

B-spline functions~\cite{peter1986curve} have properties that make them
very suitable candidates for function approximation. They can be efficiently computed in closed form based on available algorithms~\cite{de2001practical}.  B-splines are widely
employed in computer graphics, automated manufacturing, data fitting, computer graphics, and computer aided design~\cite{Farin1993133}.

A $p$th degree B-spline function, $f(t)$, defined by $n$ control
points (points that the curve passes through) and knots $\hat{t}_i$, $i=1,2,\ldots,n + p + 1$, is given by 
\begin{equation} \label{cscs}
f(t) = \sum_{i=1}^n \beta_i Q_{i,p}(t).
\end{equation}
Knot vectors are sets of non-decreasing real numbers. The
spacing between knots defines the shape of the curve along with
the control points. Function $Q_{i,p}(t)$ is called $i$th B-spline basis function of order $p$ and it can be described by the recursive equations 
\begin{equation}\label{dada}
Q_{i,0}(t) = \begin{cases} 1 & t \in [\hat{t}_i, \hat{t}_{i+1}) \\ 0 & \text{otherwise}  \end{cases}
\end{equation}
and
\begin{equation}\label{eqwew}
Q_{i,p}(t) = \frac{t - \hat{t}_i}{\hat{t}_{i+p}-\hat{t}_i}Q_{i,p-1}(t) + \frac{\hat{t}_{i+p+1} - t}{\hat{t}_{i+p+1}-\hat{t}_{i+1}} Q_{i+1,p-1}(t),
\end{equation}
 defined using the Cox-de
Boor algorithm \cite{de2001practical}. First-order basis functions are evaluated
using~\eqref{dada}, followed by iterative evaluation of \eqref{eqwew} until
the desired order is reached. In contrast to B\'ezier curves, the number $n$ of control points of the curve, is independent of the order, $p$, providing more
robustness for the generated paths topology.

Furthermore, the derivative of a B-spline $Q_{i,p}$ of degree $p$  is simply a function of B-splines of degree $p-1$:
\begin{equation} \label{fscsc}
\frac{\mathrm{d}Q_{i,p}(t)}{\mathrm{d}t} = (p-1)\bigg( \frac{-Q_{i+1,p-1}}{t_{i+p}-t_{i+1}}+ \frac{Q_{i,p-1}}{t_{i+p-1}-t_{i}} \bigg).
\end{equation}

\section{Construction of Data-Driven Differential Inclusions} \label{sec:main}

In this section, we bring forward a method that uses B-spline approximation of  data and known regularity side information to construct a data-driven model. We call this model a data-driven differential inclusion. We show in the sequel that such data-driven differential inclusions enable safety analysis  and controller synthesis.

\subsection{Piecewise-Polynomial Approximation of the System Dynamics}

We study systems for which,  at  time instances $t_1\le t_2 \le \cdots \le t_N$, samples of control and state values are available. In order to construct  data-driven differential inclusions, we employ the tensor-product spline technique~\cite{GROSSE198029} which allows for efficient approximation of multi-variate systems. To this end, we parametrize the states as a function of $t$ and $u=(u_1,\ldots,u_m)^T$,   $x(t,u)$. Furthermore, in order to obtain a control-affine model, we consider  first-order spline (piecewise linear) functions of control inputs $u=(u_1,\ldots,u_m)^T$. That is, for each state, we have 
\begin{equation}
X_l(t) = \sum_{i=1}^L \sum_{j=1}^m c_{i,j}^l Q_{i,j}^l(t) (a_{0,j}^l+a_{1,j}^lu_j),~~ l=1,2,\ldots,n.
\end{equation}
We are interested in approximating the time evolution of the state. Computing the time-derivative of $X$ yields
\begin{multline} \label{eq:XXXXXZAA}
\dot{X}_l(t) = \sum_{i=1}^L \sum_{j=1}^m c_{i,j}^l \dot{Q}_{i,j}^l(t) (a_{0,j}^l+a_{1,j}^lu_j) 
\\= \sum_{i=1}^L \sum_{j=1}^m a_{0,j}^l c_{i,j}^l \dot{Q}_{i,j}^l(t) + \sum_{i=1}^L \sum_{j=1}^m a_{1,j}^l c_{i,j}^l  \dot{Q}_{i,j}^l(t) u_j,
\end{multline}
where $\dot{Q}_{i,j}^l(t)$ can be computed efficiently using the recursive formula~\eqref{fscsc}. Define
\begin{eqnarray*}
f_l(t) &=& \sum_{i=1}^L \sum_{j=1}^m a_{0,j}^l c_{i,j}^l \dot{Q}_{i,j}^l(t) \quad \text{and} \nonumber \\
G_l(t) &=& \begin{bmatrix} \sum_{i=1}^L  a_{1,1}^l c_{i,1}^l  \dot{Q}_{i,1}^l(t) \\ \vdots \\ \sum_{i=1}^L  a_{1,m}^l c_{i,m}^l  \dot{Q}_{i,m}^l(t) \end{bmatrix}^T \in \mathbb{R}^{1\times m}.  \nonumber 
\end{eqnarray*}
Then, \eqref{eq:XXXXXZAA} can be rewritten as
\begin{equation} \label{dssdsdcceeww}
\dot{X}(t) = F(t)+G(t)u,
\end{equation}
where $X= (X_1(t),\cdots,X_n(t))^T$, $F = (f_1(t),\cdots,f_n(t))^T$ and $G=(G_1(t),\cdots,G_n(t))^T$.

\subsection{Data-Driven Differential Inclusions}

We are given  samples of the state $\{x(t_i,u_i)\}_{i=1}^N$ at time instances $t_1\le t_2\le \cdots \le t_N$. That is,  information about only one trajectory  is available. We consider state evolutions that belong to $\mathcal{C}^2(\mathbb{R}_{\ge 0})$. Hence, $\dot{x}(t) \in \mathcal{C}^1(\mathbb{R}_{\ge 0})$. We consider cases in which, in addition to state and input samples, some prior regularity knowledge (that we also call side information) on the state evolutions  may be available. We capture such information by constraints in the form of 
$$
\| x \|_{\mathcal{C}^2} \le M,
$$
for a constant $M>0$. For mechanical systems, for example, the above constraint represents  bounds on maximum acceleration. 

In order to account for the uncertainty in approximating $\dot{x}$ with the function $\dot{X}$,  we introduce  a function $\sigma:\mathbb{R}_{\ge 0} \to \mathbb{R}_{\ge 0}$ such that
 \begin{equation}\label{eq:sidenew}
 |\dot{X}(t) - \dot{x}(t)| \le M\sigma(t).
\end{equation}
 In the case of piecewise-polynomial approximation of the data, we have~\eqref{dssdsdcceeww}. From the side information~\eqref{eq:sidenew}, we have $\dot{X}_{-} \le \dot{x} \le \dot{X}_{+}$, where $\dot{X}_{-}=\dot{X}(t)-M\sigma(t)$ and $\dot{X}_{+}=\dot{X}(t)+M\sigma(t)$ and $\dot{X}(t)$ is given as~\eqref{dssdsdcceeww}.  Hence, the dynamics of the system for $t>t_N$ can be described by the following data-driven differential inclusion
\begin{eqnarray}\label{mainDIS}
\begin{cases} \dot{x}(t) \in co \big \{ \dot{X}_{-}(t),\dot{X}_{+}(t) \big \},  \\
 x(t_N)=x_N.
 \end{cases}
 \end{eqnarray}
 Note that the above dynamics are  dependent on the control signal through~\eqref{dssdsdcceeww}.
 
 \begin{figure}[!t]
\centering
\includegraphics[width=7cm]{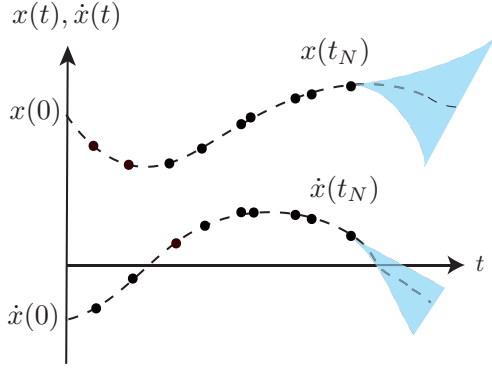}
\caption{Data samples (black dots), actual system state evolution  (dashed lines), and the solution set of the data-driven differential inclusion (blue). } \label{fig1ccc}
\end{figure}
 

 Differential inclusion~\eqref{mainDIS} in fact over-approximates the dynamics after $t>t_N$. Note that the information on the system state is only available for $0<t<t_N$ and the regularity information ($M$ and $\sigma$) provides the means using~\eqref{mainDIS} to predict the behavior of system state. The  rate of growth of the approximation error for $t>t_N$ can be approximated by the application of Gronwall's inequality~\cite{aubincelina}. This  rate is a linear function of the Lipschitz constant of the interpolant.  Nonetheless, cubic-spline polynomials are the best interpolants in terms of minimizing the Lipschitz constant for $\mathcal{C}^2$ trajectories~\cite{herbert2014computing}. Figure~\ref{fig1ccc} illustrates the solution set of the differential inclusion~\eqref{mainDIS} for a system with only one state and $\sigma(t) = 1,~~t>0$.

 We are interested in solving the following safety analysis  problem.
\\

\textbf{Problem 1:} Consider the data-driven differential inclusion~\eqref{mainDIS} with $u \equiv 0$. Given $\mathcal{X}_u \subset \mathbb{R}^n$ and $T>t_N$, check whether $x(T) \notin \mathcal{X}_u$.
\\

In addition, we are interested in addressing the following safe controller synthesis problem.
\\

\textbf{Problem 2:} Consider the data-driven differential inclusion~\eqref{mainDIS}. Given $\mathcal{X}_u \subset \mathbb{R}^n$ and $T>t_N$, find a feedback law $u(x)$ such that $x(T) \notin \mathcal{X}_u$.
\\
  
 In the next section, we discuss differential inclusions in the form of~\eqref{mainDIS} and we address Problem 1 and Problem 2.

\section{Safety analysis  and Safe Controller Synthesis for Differential Inclusions}\label{sec:controlsafe}

We start by deriving a safety theorem based on barrier certificates for convex differential inclusions and then  show how this result can be applied to the data-driven differential inclusion~\eqref{mainDIS} for safe controller synthesis.

Let $\{f_i(t,x,u)\}_{i=1}^m$ be a family of (piecewise) smooth functions, where $f_i:\mathcal{T} \times \mathcal{X} \times \mathcal{U}  \to \mathbb{R}^n$ with $\mathcal{X} \subseteq \mathbb{R}^n$, $\mathcal{T} \subseteq \mathbb{R}_{\ge0}$, and $\mathcal{U} \subseteq \mathbb{R}^m$.  We further assume $u \in \mathcal{L}^\infty$. Define $\mathcal{F}: \mathcal{T} \times \mathcal{X} \times \mathcal{U} \to 2^{\mathbb{R}^n}$ with
$$
\mathcal{F}(t,x,u) = co \{ f_1(t,x,u),\ldots, f_m(t,x,u) \}.
$$
Consider the following differential inclusion
\begin{eqnarray} \label{eq:DImain}
\begin{cases}
\dot{x} \in \mathcal{F}(t,x,u),\quad t \ge t_0, \label{eq:DI}  \\
x(t_0) = x_0.
\end{cases}
\end{eqnarray}


Well-posedness conditions of differential inclusions~\cite[Theorem 1, p. 106]{fillipov} require the set-valued map $\mathcal{F}$ to be closed and convex for all $t$, $x$, and $u$, and also measurable in $u$. The set $\mathcal{F}(t,x,u)$ is closed and convex, because it is defined as the convex hull of a finite set.   Furthermore, the mapping $\mathcal{F}(t, \cdot,u) : \mathcal{X}\to 2^{\mathbb{R}^n}$ is upper hemi-continuous in $x$, for all $t$ and $u$, because it can be written as a convex combination of smooth mappings $f_i(t,\cdot,u)$.
Finally, the mapping $\mathcal{F}(t,x,u)$ satisfies the one-sided Lipschitz condition
\begin{equation*}
(x_1-x_2)^T \left(v_1 - v_2 \right) \le C \| x_1-x_2\|^2, \quad \forall t>0,
\end{equation*}
for some $C>0$ and all $x_1 \in \mathcal{X}$, $x_2 \in \mathcal{X}$, $v_1 \in \mathcal{F}(t,x_1,u)$ and $v_2 \in \mathcal{F}(t,x_2,u)$, which follows from the fact that $\mathcal{F}(t,\cdot,u)$ is a convex hull of smooth and  Lipschitz functions for all $t$ and $u$. 

\subsection{Safety Analysis for Autonomous Differential Inclusions} \label{sec:ddss|}

In order to propose a solution to Problem 1,  we  extend the concept of barrier certificates to differential inclusions. 

%

 Before stating the result, we require a definition of the derivative for set-valued maps~\cite{aubinsetvalued}. Denote by
$$
D_{+} V(x)(f) = \liminf_{h \to 0^{+},~v \to f} \frac{V(x+hv)-V(x)}{h}
$$
the upper contingent derivative of $V$ at $x$ in the direction $f$. In particular, when $V$ is Gateaux differentiable and $\mathcal{F}=\{v\}$ is a singleton, $D_{+}V(x)$ coincides with the gradient 
$$
D_{+}V(x)(v) = \left(\nabla V(x)\right)^T v.
$$  

\begin{thm} \label{thm:barriersDifInq}
Consider differential inclusion~\eqref{eq:DI} with $u \equiv 0$ and let $T>t_0$. If there exist a function $B \in \mathcal{C}^1(\mathbb{R}^n;\mathbb{R}) \cap \mathcal{C}^1(\mathbb{R}_{\ge 0};\mathbb{R})$ and a positive definite function $W\in \mathcal{L}(\mathbb{R}^n \times \mathbb{R}_{\ge 0};\mathbb{R}_{\ge 0})$ such that 
\begin{equation}\label{eq:Bcon1}
B\left(x(T),T\right)-B\left(x(t_0),t_0\right) >0, \quad x(T) \in \mathcal{X}_u,
\end{equation}
\begin{equation} \label{eq:Bcon2}
D_{+} B(t,x)(v,1) \le -W(t,x), \quad t \in [t_0,T],~~v \in \mathcal{F}(t,x),
\end{equation}
then the solutions of \eqref{eq:DI} satisfy $x(T) \notin \mathcal{X}_u$.
\end{thm}

\begin{proof}
The proof is carried out by contradiction. Assume it holds that $x(T) \in \mathcal{X}_u$. Then, \eqref{eq:Bcon1} implies that 
$$
B\left(x(T),T\right)>B\left(x(t_0),t_0\right).
$$
Furthermore, using the comparison theorem for differential inclusions~\cite[Proposition~8, p. 289]{aubincelina} and inequality~\eqref{eq:Bcon2}, we can infer that
$$
B(x(s),s)-B(x(t_0),t_0) \le - \int_{t_0}^s W(x,\tau)~d\tau \le 0.
$$
That is, 
$$
B(x(s),s) \le B(x(t_N),t_N).
$$
Since $s$ was chosen arbitrary, this is a contradiction. Thus, the solutions of \eqref{eq:DI} satisfy $x(T) \notin \mathcal{X}_u$.
\end{proof}

Theorem~\ref{thm:barriersDifInq} presents conditions to check the safety of the solutions of a differential inclusion. The above formulation only necessitates the barrier certificate to have  an upper contingent derivative. Hence, our formulation here includes non-smooth barrier certificates. The closet work on this class of barrier certificates is the recent research~\cite{7937882}, wherein the authors considered min/max barrier functions with application to control of multi-robot systems. Nonetheless, the formulation in~\cite{7937882} is based on fixing the the barrier functions, i.e., the 0-super level sets of the barrier functions define the  unsafe sets. This assumption can be restrictive in many cases, in particular, general semi-algebraic sets. In this study, we do not make such assumptions on the structure of the barrier functions and we allow them to be variables in the computational formulation in Section~\ref{sec:compute}.


\subsection{Safety Controller Synthesis for Data-Driven Differential Inclusions} \label{dccsdccsec}

 Having established a safety analysis theorem for differential inclusions,  we proceed by proposing conditions for safety analysis and safe controller synthesis  for  data-driven differential inclusion~\eqref{mainDIS}. We begin with the safety analysis result, which follows from Theorem~\ref{thm:barriersDifInq}. In other words, given the limited data over states up to some time $t_N$ and the regularity information, we verify whether the system behaves safely at a given time $T>t_N$. 

\begin{cor}\label{cordsdsds}
Consider differential inclusion~\eqref{mainDIS} with $u \equiv 0$ and let $T>t_N$. If there exist a function $B \in \mathcal{C}^1\left(\mathbb{R}^n;\mathbb{R}) \cap \mathcal{C}^1([t_N,\infty);\mathbb{R}\right)$ and two positive definite functions $W_{+} \in \mathcal{L}\left(\mathbb{R}^n \times [t_N,\infty);[t_N,\infty)\right)$ and $W_{-} \in \mathcal{L}\left(\mathbb{R}^n \times [t_N,\infty);[t_N,\infty)\right)$ such that 
\begin{equation}\label{eq:Bcon11}
B\left(x(T),T\right)-B\left(x(t_N),t_N\right) >0, \quad x(T) \in \mathcal{X}_u,
\end{equation}
\begin{eqnarray} 
D_{+} B(t,x)\left(\dot{X}_{-},1\right) \le -W_{-}(t,x), \quad t \in [t_N,T], \label{eq:Bcon22} \\
D_{+} B(t,x)\left(\dot{X}_{+},1\right) \le -W_{+}(t,x), \quad t \in [t_N,T], \label{eq:Bcon33}
\end{eqnarray}
then the solutions of \eqref{mainDIS} satisfy $x(T) \notin \mathcal{X}_u$.
\end{cor}

\begin{proof}
Inequality~\eqref{eq:Bcon11} ensures that ~\eqref{eq:Bcon1} holds.
Multiplying both sides of inequality~\eqref{eq:Bcon22} with a constant $0 \le \alpha_{-}\le 1$ and inequality~\eqref{eq:Bcon33} with a constant $0 \le \alpha_{+}\le 1$  such that $ \alpha_{-}+\alpha_{+} =1$ and adding them, we obtain
\begin{multline*}
 \alpha_{-}  D_{+} B(t,x)(\dot{X}_{-},1) +\alpha_{+}  D_{+} B(t,x)(\dot{X}_{+},1)  
\\  \le - \alpha_{-} W_{-}(t,x)-\alpha_{+} W_{+}(t,x).
\end{multline*}
Since $D_{+}$ is a linear operator, we have
\begin{multline*}
D_{+} B(t,x)( \alpha_{-}\dot{X}_{-}+ \alpha_{+}\dot{X}_{+},1) \\ \le D_{+} B(t,x)( \alpha_{-}\dot{X}_{-}) + D_{+} B(t,x)( \alpha_{-}\dot{X}_{+})   \\ \le - \alpha_{-} W_{-}(t,x)-\alpha_{+} W_{+}(t,x),
\end{multline*}
where, in the last line above, we applied inequalities~\eqref{eq:Bcon22} and \eqref{eq:Bcon33}. Let $W(t,x) = \min \left\{ W_{-}(t,x),W_{+}(t,x) \right\}$. We obtain
\begin{multline*}
D_{+} B(t,x)(\alpha_{-}\dot{X}_{-}+ \alpha_{+}\dot{X}_{+},1) \\ \le - \alpha_{-} W_{-}(t,x)-\alpha_{+} W_{+}(t,x) \\ \le -(\alpha_{-}+\alpha_{+}) W(t,x) = -W(t,x).
\end{multline*}
That is, 
$$
D_{+} B(t,x)(v,1) \le -  W(t,x),\quad v \in co\{\dot{X}_{-},\dot{X}_{+}\}.
$$
Thus inequality~\eqref{eq:Bcon2} is also satisfied. This completes the proof.
\end{proof}

We show in Section~\ref{sec:compute} that if we parametrize $B$ and $W$ by polynomials of fixed degree and let $\mathcal{X}_u$ be a semi-algebraic set, then we can check inequalities~\eqref{eq:Bcon11} through~\eqref{eq:Bcon33} using sum of squares programming.


We next address Problem 2 by proposing a  method to design controllers such that solutions to~\eqref{dssdsdcceeww} are safe with respect to a given unsafe set $\mathcal{X}_u$  for all $t>t_N$. In other words, we address Problem 2.

\begin{cor}\label{cordsdsds}
Consider differential inclusion~\eqref{mainDIS} and let $T>t_N$. If there exist a function $B \in \mathcal{C}^1\left(\mathbb{R}^n;\mathbb{R}) \cap \mathcal{C}^1([t_N,\infty);\mathbb{R}\right)$ and a positive definite function $W\in \mathcal{L}\left(\mathbb{R}^n \times [t_N,\infty);[t_N,\infty)\right)$ such that 
\begin{equation}\label{eq:Bcon12221}
B\left(x(T),T\right)-B\left(x(t_N),t_N\right) >0, \quad x(T) \in \mathcal{X}_u,
\end{equation}
\begin{multline} \label{rwewdw2222wcv}
\frac{\partial B}{\partial t} +  \left( \frac{\partial B}{\partial x} \right)^T( f-M\sigma) \le  -W(t,x),\\ \forall x \in \mathcal{X},~~\forall t \in [t_N,T],
\end{multline}
\begin{multline} \label{cccxzeew}
\frac{\partial B}{\partial t} +  \left( \frac{\partial B}{\partial x} \right)^T( f+M\sigma) \le  -W(t,x),  \\ \quad \forall x \in \mathcal{X},~~\forall t \in [t_N,T],\end{multline}
and
\begin{equation} \label{eq:controlability}
\left( \frac{\partial B}{\partial x} \right)^T g \neq 0, \quad \forall x \in \mathcal{X},~~\forall t \in [t_N,T],
\end{equation}
then the  controller
\begin{equation} \label{eq:controller}
u =g^T \frac{\partial  B}{\partial x} \left( (\frac{\partial  B}{\partial x})^T g g^T \frac{\partial  B}{\partial x} \right)^{-1} W(t,x),
\end{equation}
renders the solutions of \eqref{mainDIS} safe, i.e., $x(t) \notin \mathcal{X}_u$ for all $t \in [t_N,T]$.
\end{cor}
\begin{proof}
The proof follows by applying Corollary~\ref{cordsdsds} to system~\eqref{dssdsdcceeww} with $W_{-}=W_{+}=W$. Inequality~\eqref{eq:Bcon12221} ensures that~\eqref{eq:Bcon11} holds. Computing the directional derivative of $B$ along the solutions of system~\eqref{dssdsdcceeww} and using the linearity property of the $D_{+}$ operator  we have 
\begin{multline} \label{cscscsC}
D_{+} B(t,x)\left(f(t)+g(t)u \pm M\sigma(t),1\right) \\ = \frac{\partial B}{\partial t} +  \left( \frac{\partial B}{\partial x} \right)^T( f+gu \pm M\sigma)  
\\= \frac{\partial B}{\partial t} +  \left( \frac{\partial B}{\partial x} \right)^T( f \pm M\sigma) + \left( \frac{\partial B}{\partial x} \right)^Tgu.
\end{multline}
Note that if $\left( \frac{\partial B}{\partial x} \right)^Tg = 0$, we cannot design  $u$ such that makes the last line of~\eqref{cscscsC} negative definite, i.e.,  to ensure safety. Noting that~\eqref{eq:controlability} holds and substituting the controller~\eqref{eq:controller} in the last line of~\eqref{cscscsC}, we obtain
\begin{multline} \label{opoepeore}
 \frac{\partial B}{\partial t} +  \left( \frac{\partial B}{\partial x} \right)^T( f \pm M\sigma) 
 \\+ \left( \frac{\partial B}{\partial x} \right)^Tg g^T \frac{\partial  B}{\partial x} \left( (\frac{\partial  B}{\partial x})^T g g^T \frac{\partial  B}{\partial x} \right)^{-1} W(t,x) \\
 =  \frac{\partial B}{\partial t} +  \left( \frac{\partial B}{\partial x} \right)^T( f \pm M\sigma) \\+ W(t,x) -W(t,x)+W(t,x) \le 0,
\end{multline}
where, in the last line inequality above, we used the fact that~\eqref{rwewdw2222wcv} and \eqref{cccxzeew} hold. Thus, 
$$
D_{+} B(t,x) ( f(t)+g(t) \pm M\sigma(t),1) \le 0.
$$ 
Let $0 \le \alpha_1,\alpha_2 \le 1$ satisfying $\alpha_1+\alpha_2=1$. Then, 
\begin{multline*}
\alpha_1 D_{+} B(t,x) ( f(t)+g(t) + M\sigma(t),1) 
\\+\alpha_2 D_{+} B(t,x) ( f(t)+g(t) - M\sigma(t),1) \\
= D_{+} B(t,x) \left( \alpha_1(f(t)+g(t) + M\sigma(t)),1 \right) \\+ D_{+} B(t,x) \left( \alpha_2(f(t)+g(t) - M\sigma(t)),1 \right) \\
= D_{+} B(t,x) \bigg( \alpha_1(f(t)+g(t) + M\sigma(t)) 
\\+\alpha_2(f(t)+g(t) - M\sigma(t)) ,1 \bigg) \le 0,
\end{multline*}
where in the last line we used ~\eqref{opoepeore}. This completes the proof.
\end{proof}



\section{Computational Method} \label{sec:compute}

In this section, we propose  computational methods to  address the analysis  and the synthesis problems. Piecewise polynomial interpolation leads to a computational formulation based on  polynomial optimization or  sum-of-squares 
 programs.

Assuming $\sigma \in \Sigma[t]$,  \eqref{mainDIS} becomes a differential inclusion with polynomial vector fields. The next lemma, which is based on the application of Putinar's Positivestellensatz~\cite{putinar99,S1052623403431779}, presents  conditions in terms of polynomial positivity that can be efficiently checked via semi-definite programs (SDPs)~\cite{Par00}. 
\begin{lem}\label{lem1}
Consider the differential inclusion~\eqref{mainDIS} with $u \equiv 0$ and the following semi-algebraic unsafe set 
\begin{equation} \label{dcscsdcscscs}
\mathcal{X}_u = \{ x \mid l_i(x) \le 0,~i=1,2,...,n_c\},
\end{equation}
where $l_i \in \mathcal{R}[x]$. If there exist functions $B \in \mathcal{R}[x,t]$, $W_{-} \in {\Sigma}[x,t]$, $W_{+} \in {\Sigma}[x,t]$, $m_i \in \Sigma[x,t]$, $i=1,2$, $s_i \in \Sigma[x,t]$, $i=1,\ldots,n_c$ and a positive constant $c>0$, such that
\begin{multline} \label{Con1111s}
B\left(x(T),T\right)-B\left(x(t_N),t_N\right)  \\+ \sum_{i=1}^{n_c} s_i\left(x(T)\right) l_i\left(x(T)\right) -c \in \Sigma \left[x(T)\right]
\end{multline}
and
\begin{multline} \label{rwewdwwcv1}
-\frac{\partial B}{\partial t} - \left( \frac{\partial B}{\partial x} \right)^T \dot{X}_{-}   -W_{-}(t,x) \\- m_1(t,x)(t-t_N)(t-T) \in  \Sigma[x,t],
\end{multline}
\begin{multline} \label{wqqwsszz}
-\frac{\partial B}{\partial t} - \left(\frac{\partial B}{\partial x}\right)^T \dot{X}_{+}  -W_{+}(t,x) \\- m_2(t,x)(t-t_N)(t-T) \in  \Sigma[x,t],
\end{multline}
then the solutions to~\eqref{mainDIS} satisfy $x(T) \notin \mathcal{X}_u$. 
\end{lem}
\begin{proof}
Applying Putinar's Positivstellensatz, condition~\eqref{Con1111s} implies that 
$$
B\left(x(T),T\right)-B\left(x(t_N),t_N\right)>0,
$$
for all $x(T) \in \mathcal{X}_u$ as in~\eqref{eq:Bcon11}. Thus, inequality~\eqref{eq:Bcon11} holds. Moreover, since $B \in \mathcal{R}[x,t]$ and thus smooth, we have
$$
D_{+} B(t,x)\left(\dot{X}_{-},1\right) = \frac{\partial B}{\partial t} + \left( \frac{\partial B}{\partial x}\right)^T \dot{X}_{-} .
$$
Hence, from condition~\eqref{rwewdwwcv1}, we have
 $$
 D_{+} B(t,x)\left(\dot{X}_{-},1\right) \le -W_{-}(t,x), \quad t \in [t_N,T].
 $$
 Therefore, inequality~\eqref{eq:Bcon22} is satisfied. In a similar manner, we can show that \eqref{eq:Bcon33} holds as well. Analogously, we can show that~\eqref{wqqwsszz} implies that \eqref{eq:Bcon33} is satisfied. Then, from Corollary~\ref{cordsdsds}, the solutions to~\eqref{mainDIS} satisfy $x(T) \notin \mathcal{X}_u$.
\end{proof}

We now focus on formulating sum of squares conditions to synthesize a safe controller as highlighted in Corollary 2. The following result can be proved by a direct application of Positivstellensatz. 

\begin{lem}\label{lem5651}
Consider the differential inclusion~\eqref{mainDIS} and the following semi-algebraic unsafe set 
\begin{equation} \label{dcscsdcscscs}
\mathcal{X}_u = \{ x \mid l_i(x) \le 0,~i=1,2,...,n_c\},
\end{equation}
where $l_i \in \mathcal{R}[x]$. If there exist functions $B \in \mathcal{R}[x,t]$, $W \in {\Sigma}[x,t]$, $m_i \in \Sigma[x,t]$, $i=1,2$, $s_i \in \Sigma[x,t]$, $i=1,\ldots,n_c$, $q_1 \in  \mathcal{R}[x,t]$, $q_2 \in \Sigma[x,t]$ and a positive constant $c>0$ such that
\begin{multline} \label{Con1111}
B\left(x(T),T\right)-B\left(x(t_N),t_N\right)  \\+ \sum_{i=1}^{n_c} s_i\left(x(T)\right) l_i\left(x(T)\right) -c \in \Sigma \left[x(T)\right],
\end{multline}
\begin{multline} \label{rwewdwwcv}
-\frac{\partial B}{\partial t} - \left( \frac{\partial B}{\partial x} \right)^T( f-M\sigma) 
\\-W(t,x) - m_1(t,x)(t-t_N)(t-T) \in  \Sigma[x,t],
\end{multline}
\begin{multline}
-\frac{\partial B}{\partial t} - \left( \frac{\partial B}{\partial x} \right)^T( f+M\sigma) \\-W(t,x) - m_2(t,x)(t-t_N)(t-T) \in  \Sigma[x,t],
\end{multline}
and
\begin{equation} \label{eq:bilinearB}
-q_1  \left( \left( \frac{\partial B}{\partial x} \right)^T g \right) - q_2 \in \Sigma[x,t],
\end{equation}
then the solutions to~\eqref{mainDIS} satisfy $x(T) \notin \mathcal{X}_u$. 
\end{lem}

Note that sum of squares constraint~\eqref{eq:bilinearB} is bilinear in variables $B$ and $q_1$. In practice, we fix the polynomial $q_1$ and look for a barrier certificate such that \eqref{Con1111} to \eqref{eq:bilinearB} are satisfied. In many cases such as the examples given in Section~\ref{sec:numresults}, this constraint holds for a computed barrier certificate $B$. 

\section{Including Side Information: Algebraic and Integral Inequalities} \label{sec:side}

In addition to the regularity side information discussed so far, we consider physical conservation laws in the form of algebraic and integral inequalities (note that we can always represent an equality with two inequalities). The side information can be used to restrict the set of predictions we make regarding the dynamics of the system at a future point in time.  For instance, consider the two-dimensional motion of a simple pendulum  of fixed length $\ell>0$ and a bob of mass $m$ (see Figure~\ref{fig:pendul}). Let $x$ and $y$ represent the horizontal and vertical positions of the bob, which also account for the states of the system. Since the length of pendulum is fixed (non-elastic), the states always satisfy the following equality (conservation of length)
$$
x^2+y^2= \ell^2.
$$
In the following, we show how we can use such physical side information to obtain less conservative predictions. Formally, we consider  side information that can be represented as
\begin{multline}\label{hardiqcs}
\mathcal{C} = \bigg \{ (t,x,u) \mid \phi_i(t,x,u)\ge0,~i \in I, \\
\int_0^T \psi_j(t,x,u)~dt \ge0,~j \in J,~T>0 \bigg\},
\end{multline}
where $I$ and $J$ are index sets and $\phi_i,~i \in I$ and $\psi_j,~j \in J$ are continuous functions. 

\begin{figure}
  \centering
\begin{tikzpicture}[thick,>=latex,->,scale=0.7]

\begin{scope}
\clip(-5,2) rectangle (5,-5);

\draw[dashed] (0,0)  circle (4.24cm);
\filldraw[white] (-4.3,4.3) rectangle (4.3,0);
\draw[double distance=1.6mm] (0,0) -- (3,-3) node[midway,xshift=4mm,yshift=2mm]{$\ell$};
\draw[->] (3,-3) -- (3,-4.5) node[below]{$m g$};
\draw[fill=white] (-1.2,1.0) -- (-.5,0) arc(180:360:0.5) -- (1.2,1.0) -- cycle;
\draw[draw=black,fill=white] (0, 0) circle circle (.3cm);
\draw[draw=black,fill=white] (3,-3) circle circle (.3cm);
\draw[->] (.6,0) -- (2,0) node[below]{{$x$}};
\draw[->] (0,-.6) -- (0,-2) node[below]{{$y$}};
\draw[pattern=north east lines] (-1.4,1.3) rectangle (1.4,1);
\end{scope}

\end{tikzpicture}
\caption{Schematic diagram of the two dimensional pendulum.} \label{fig:pendul}
\end{figure}
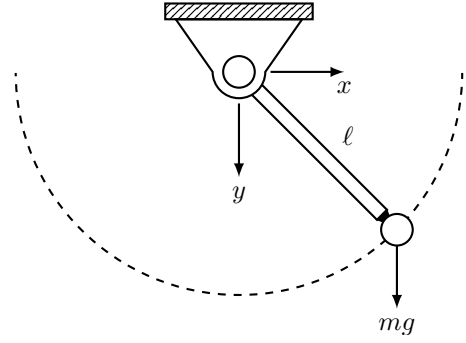

In the case in which $\psi_j$'s are  quadratic functions of $x$ and/or $u$, they can represent (hard) integral quadratic constraints, which can be used to model system properties such as passivity, induced input-output norms, saturation and etc. In the absence of inputs, we have
\begin{multline} \label{eq:sideauton}
\mathcal{C}_0 = \bigg \{ (t,x) \mid \phi_i(t,x)\ge 0,~i \in I, \\
\int_0^T \psi_j(t,x)~dt\ge0,~j \in J,~T>0 \bigg\}.
\end{multline}

We can incorporate  side information in the form of~\eqref{eq:sideauton}  as additional constraints. To illustrate, conditions of Corollary~\ref{cordsdsds} can be rewritten~as 

\begin{equation}\label{eq:Bcon11cccx}
B\left(x(T),T\right)-B\left(x(t_N),t_N\right) >0, \quad x(T) \in \mathcal{X}_u,
\end{equation}
\begin{equation} \label{eq:Bconcxssssccxc22}
D_{+} B(t,x)\left(\dot{X}_{-},1\right) \le -W_{-}(t,x),~~ t \in [t_N,T],~~(t,x) \in \mathcal{C}_0,  
\end{equation}
and
\begin{equation} 
D_{+} B(t,x)\left(\dot{X}_{+},1\right) \le -W_{+}(t,x),~~ t \in [t_N,T],~~(t,x) \in \mathcal{C}_0.
\end{equation}

The next result gives a sum of squares formulation to the above inequalities and therefore allows us to assess the safety of the system given the side information. 

\begin{lem}\label{lem2}
Consider the differential inclusion~\eqref{mainDIS} with $u \equiv 0$, the semi-algebraic unsafe set
$$
\mathcal{X}_u = \{ x \mid l_i(x) \le 0,~i=1,2,...,n_c\},
$$
 and side information~\eqref{eq:sideauton} where $\phi_i \in \mathcal{R}[x,t],~i \in I$, $\psi_j\in \mathcal{R}[x,t],~j\in J$, and $l_i \in \mathcal{R}[x]$. If there exist functions $B \in \mathcal{R}[x,t]$, $W_{-} \in {\Sigma}[x,t]$, $W_{+} \in {\Sigma}[x,t]$, $m_i \in \Sigma[x,t]$, $i=1,2$, $s_i \in \Sigma[x,t]$, $i=1,\ldots,n_c$, $p_i, \hat{p}_i \in \Sigma[x,t]$, $i=1,\ldots,n_s$, and positive constants $q_i,\hat{q}_i$, $i=1,\ldots,n_{si}$ and $c>0$ such that
\begin{multline} \label{Con11211}
B\left(x(T),T\right)-B\left(x(t_N),t_N\right)  \\+ \sum_{i=1}^{n_c} s_i\left(x(T)\right) l_i\left(x(T)\right) -c \in \Sigma \left[x(T)\right]
\end{multline}
and
\begin{multline} \label{rwewdw2wcv}
-\frac{\partial B}{\partial t} - \left( \frac{\partial B}{\partial x}\right)^T \dot{X}_{-}  -W_{-}(t,x) - m_1(t,x)(t-t_N)(t-T) \\-\sum_{i=1}^{n_s} p_i(t,x) \phi_i(t,x)  -\sum_{j=1}^{n_{si}} q_j\psi_j (t,x) \in  \Sigma[x,t],
\end{multline}
\begin{multline}
-\frac{\partial B}{\partial t} - \left( \frac{\partial B}{\partial x}\right)^T \dot{X}_{+}  -W_{+}(t,x) - m_2(t,x)(t-t_N)(t-T) 
\\-\sum_{i=1}^{n_s} \hat{p}_i(t,x) \phi_i(t,x)  -\sum_{j=1}^{n_{si}} \hat{q}_j\psi_j (t,x) \in  \Sigma[x,t],
\end{multline}
then the solutions to~\eqref{mainDIS} satisfy $x(t) \notin \mathcal{X}_u$, for all $t \in [t_N,T]$. 
\end{lem}

Note that the side information requires inequalities~\eqref{eq:Bconcxssssccxc22} to hold in a smaller set; therefore, it makes it easier to find the barrier function. However, from a computational standpoint, as the number of inequalities in~\eqref{eq:sideauton} increases, we have to add more variables to the sum of squares program. Similarly, we can formulate sum of squares conditions for controller synthesis with side information~\eqref{hardiqcs} as described in Lemma~\ref{lem5651}.

\section{Numerical Results} \label{sec:numresults}

In this section, we illustrate the proposed method using two examples. The first example is a single state system for which limited data is available and safety in a future time is of interest. The second example is a safe-landing scenario for an aircraft  with critical failure. In the following examples, we used the parser  SOSTOOLs~\cite{PAVPSP13} to cast the polynomial inequalities into semidefinite programs and then we used Sedumi~\cite{Stu98}  to solve the resultant SDPs.

\subsection{Example I}

We consider 20 samples of the solution to the following differential equation
\begin{eqnarray}\label{dccsxsew}
\dot{x}& =& 0.5x^2-0.05x^3+u, \nonumber \\ 
x(0) &=& 1,
\end{eqnarray}
in the interval $0<t<1.7$ when $u \equiv 0$ (see Figure~\ref{L2curve}). The regularity information is given as $\|x\|_{\mathcal{C}^2} \le 5$ and 
$$
|\dot{X}(t) - \dot{x}(t) | \le 5,
$$
which implies that $M=5$ and $\sigma(t)=1$. We use a cubic piecewise-polynomial approximation  of $x(t)$. This can be carried out readily by the \texttt{spline} function in MATLAB. 

The unsafe set  is given by 
$$
\mathcal{X}_u = \left\{ x \in \mathbb{R} \mid x-9 \ge 0    \right\}.
$$
The boundary of the unsafe set $x=9$, the data points $\{x_i\}_{t=1}^{20}$ and the piecewise-polynomial approximation of the state $X(t) = \sum_i \beta_i Q_{i,p}(t)$ are shown in Figure~\ref{L2curve}. As it can be observed from the figure, since there exists an stable equilibrium at $x=10$, the solution of the actual system converges to the equilibrium at $t \approx 3$. However, since the piecewise-polynomial approximation of the state is based on the information up to $t=1.7$, $X(t)$ differs from $x(t)$ as time passes. Nonetheless, the data-driven differential inclusion~\eqref{mainDIS} provides an approximation of the state evolutions for $t>1.7$.

\begin{figure}
  \centering
  \includegraphics[scale=.5]{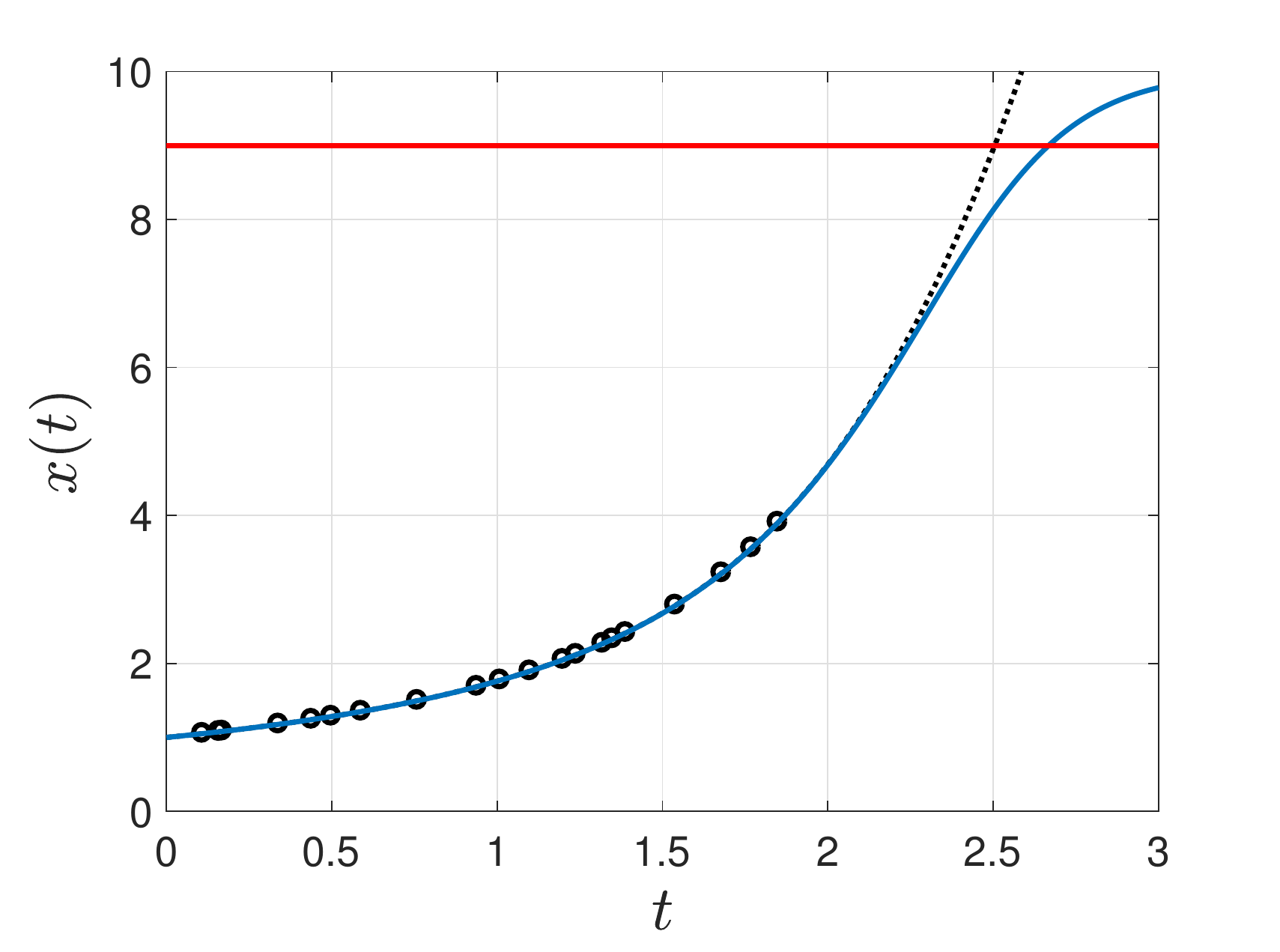}\\
  \caption{The boundary of the unsafe set $x=9$ (red line), the data points $\{x(t_i)\}_{i=1}^{20}$ (black circles), the piecewise-polynomial approximation of the state $X(t)$ (black dots) and the actual solution of the system (solid blue). }\label{L2curve}
\end{figure}

\begin{table}[!b]
\caption{Numerical results for 20 samples and $t_N=2$.}
\label{t11}
\centering
\begin{tabular}{c||c|c|c|c|c|c}

\bfseries deg & 1 & 2 & 3 & 4 & 5 & 6  \\
\hline
\bfseries $T$ & 2.26 & 2.34 & 2.41 & 2.45 & 2.46 & 2.49   \\
\end{tabular}
\end{table}

 In this example, we are interested in finding the maximum $T$ for which the solutions become unsafe, i.e., $x(T) \ge 9$. To this end, based on Corollary~\ref{lem1}, we increase the value of $T$ and look for a barrier certificate. We continue until no barrier certificate can be found. Table~\ref{t11} provides the numerical results. Notice that the actual system become unsafe at $T=2.66$. However, due to system uncertainty and limited data, the lower bound on the unsafe set has been found  to be $T=2.49$ corresponding to certificates of degree 6. The barrier certificate of degree 3 is given bellow
\begin{multline}
B(t,x) = -0.496t^3+0.119t^2x + 0.0449t^2 - 0.0383tx^2 
          \\-0.5855tx - 0.8398t + 0.1063x^2 +1.389x. \nonumber
\end{multline}

At this point, we consider 20 samples of the solution to the differential equation~\eqref{dccsxsew} from $t \in [0,2]$ and we allow one seconds for the computations. Thus, the safe controller would kick in at around $t=3$.  
The values of $u$ for learning phase are drawn from $10\sin(t)$. The unsafe set given by $\mathcal{X}_u = \left\{ x \in \mathbb{R} \mid x-10 \ge 0    \right\}$ and we are interested in making sure that the system solutions remain safe from $[3,10]$, i.e., $T=10$. Figure~\ref{L2ddcurve} shows the application of the designed safe controller based on Corollary 2. As it can be observed as the safe controller starts acting on the system, the solutions become safe for $t \ge 3$. The computed certificates $B$ and $W$ of degree 2 in the synthesis problem are given below 
\begin{multline}
B(t,x) = 0.3845t^2 - 0.0002269tx - 1.594t \\+ 0.0006x^2  + 0.0001038x,  \nonumber
\end{multline}
\begin{multline}
W(t,x) = 0.1664t^2  - 0.6558t + 1.5543x^2  + 0.6664
. \nonumber
\end{multline}

\begin{figure}
  \centering
  \includegraphics[scale=.46]{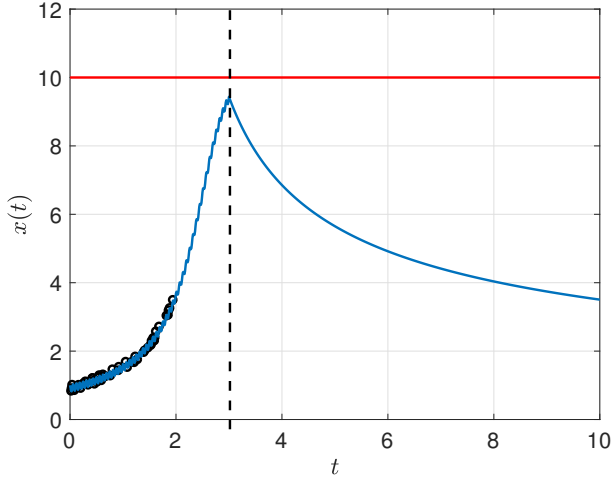}\\
  \caption{The boundary of the unsafe set $x=10$ (red line), the data points $\{x(t_i)\}_{i=1}^{20}$ (black circles),  and the  solution of the system (solid blue). The safe controller starts operating at $t=3$.}\label{L2ddcurve}
\end{figure}

\subsection{Example II: Aircraft Landing}

\begin{figure}
  \centering
  \includegraphics[scale=.4]{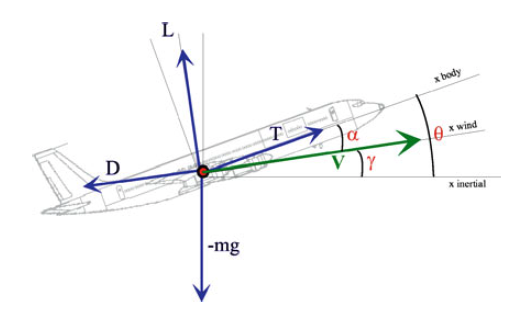}\\
  \caption{Point mass force diagram for the longitudinal dynamics of the aircraft~\cite[p. 108]{viabbookbayern}. }\label{figplane}
\end{figure}


We consider a point mass longitudinal model of an aircraft subject to the
gravity force $mg$ with $m$ being the mass and $g=9.8 ms^{-2}$, thrust $T$ , lift $L$ and drag $D$ (see Figure~\ref{figplane}).  The equations of motion are then described by
\begin{eqnarray}\label{eq:aircraft}
\dot{V} &=& \frac{1}{m} \left[ T\cos(\alpha) - D(\alpha,V)-mg\sin(\gamma) \right], \nonumber  \\
\dot{\gamma} &=& \frac{1}{mV} \left[ T\sin(\alpha) +L(\alpha,V)-mg\cos(\gamma) \right], \nonumber \\
\dot{z} &=& V\sin(\gamma),
\end{eqnarray}
where $V$  the velocity, $\gamma$  the flight path angle and $z$  the altitude are the states of the system. The angle of attack $\alpha$ and thrust $T$  are the inputs of the system. 

Typically, landing is operated at $T_{idle} = 0.2  T_{max}$ where $T_{max}$ is the
maximal thrust. This value enables the aircraft to counteract the drag due to
flaps, slats and landing gear. Most of the parameters for the DC9-30 can be
found in the literature. The values of the numerical parameters used for the
DC9-30 in this flight configuration are $m = 60, 000~(kg)$, and $T_{max} = 160, 000~(N)$. The lift and drag forces are thus
\begin{eqnarray}
L(\alpha, V ) &=& 68.6 (1.25 + 4.2\alpha)V^2, \nonumber \\
D(\alpha, V ) &=& \left[2.7+3.08 (1.25 + 4.2\alpha)^2\right]V^2,
\end{eqnarray}
both in Newtons.

For safety analysis and controller synthesis, we consider two separate aircraft failure scenarios in which the aircraft model~\eqref{eq:aircraft} is not valid anymore and we use data-driven differential inclusion models instead. In both cases, we are interested in  the safe landing speed of the aircraft $V_{safe}$ as it reaches  the ground $z=0$. Therefore, we define the unsafe set as 
$$
\mathcal{X}_u = \left\{ (V,\gamma,z) \mid V > V_{safe},~z=0 \right\}.
$$
We consider the  following regularity side information 
$$
\| x\|_{\mathcal{C}^1} \le 10,\quad |\dot{X}(t) - \dot{x}(t) | \le 10.
$$

In the first scenario, we model wing failure by sudden drop in the lift force $L$ to  $0.2L$ at $t=0.5(s)$. We collect $25$ data samples until $t=2(s)$ and we are interested in checking whether the aircraft satisfies the safe landing speed at the time of landing. Using Lemma~\ref{lem1}, we found a barrier certificate of degree $5$ proving that the system is safe. Figure~\ref{L2curve2cxcxcxc} illustrates the collected data, the actual system state evolution  and the solutions of the data-driven differential inclusion given the side information.

\begin{figure}
  \centering
  \includegraphics[scale=.39]{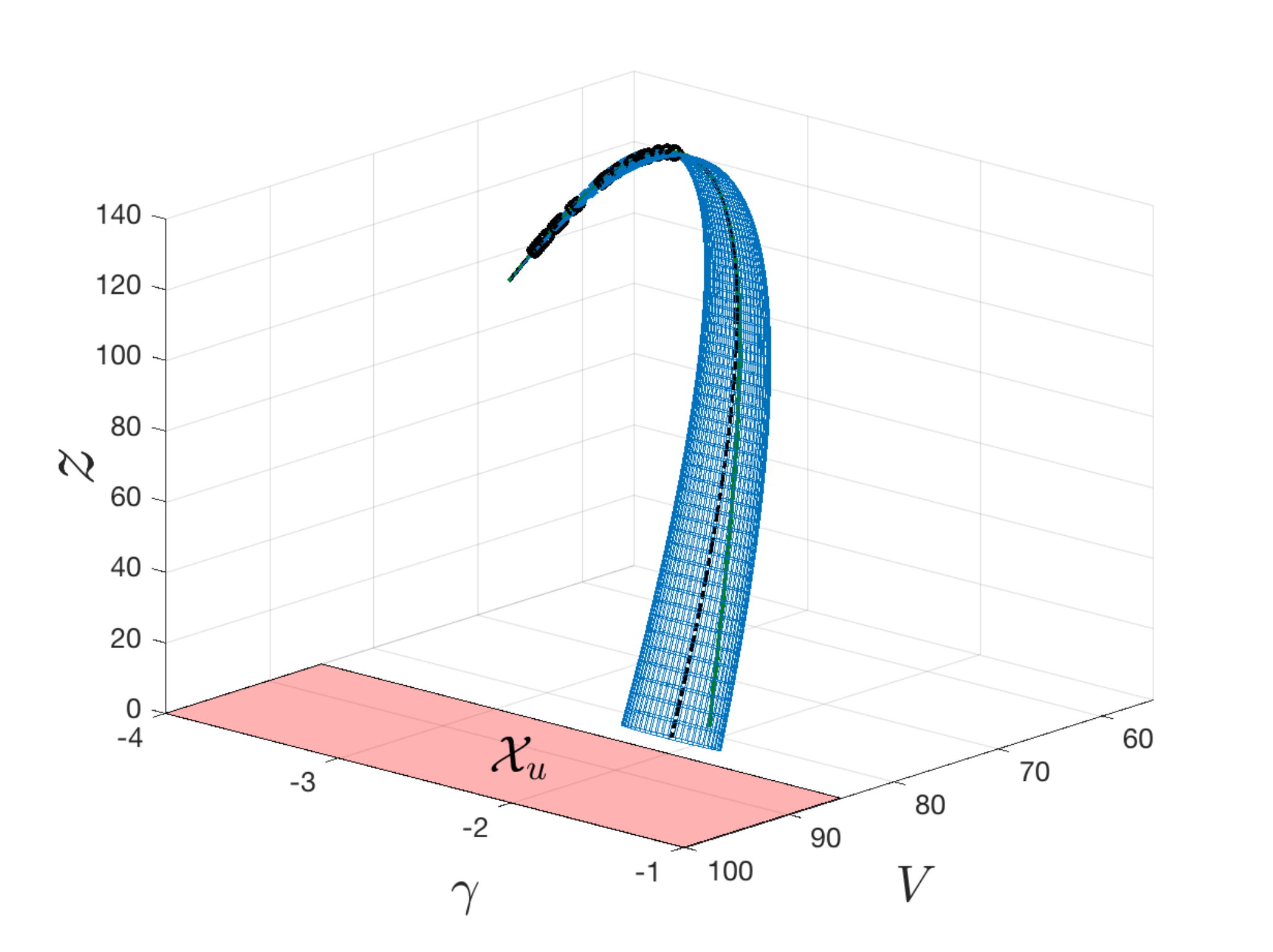}\\
  \caption{The data points $\{x(t_i)\}_{i=1}^{25}$ (black circles), the piecewise-polynomial approximation of the state $X(t)$ (black line), the solution set of the data-driven differential inclusion (the meshed cone) and the actual solution of the system (green). }\label{L2curve2cxcxcxc}
\end{figure}

\begin{figure}
  \centering
  \includegraphics[scale=.45]{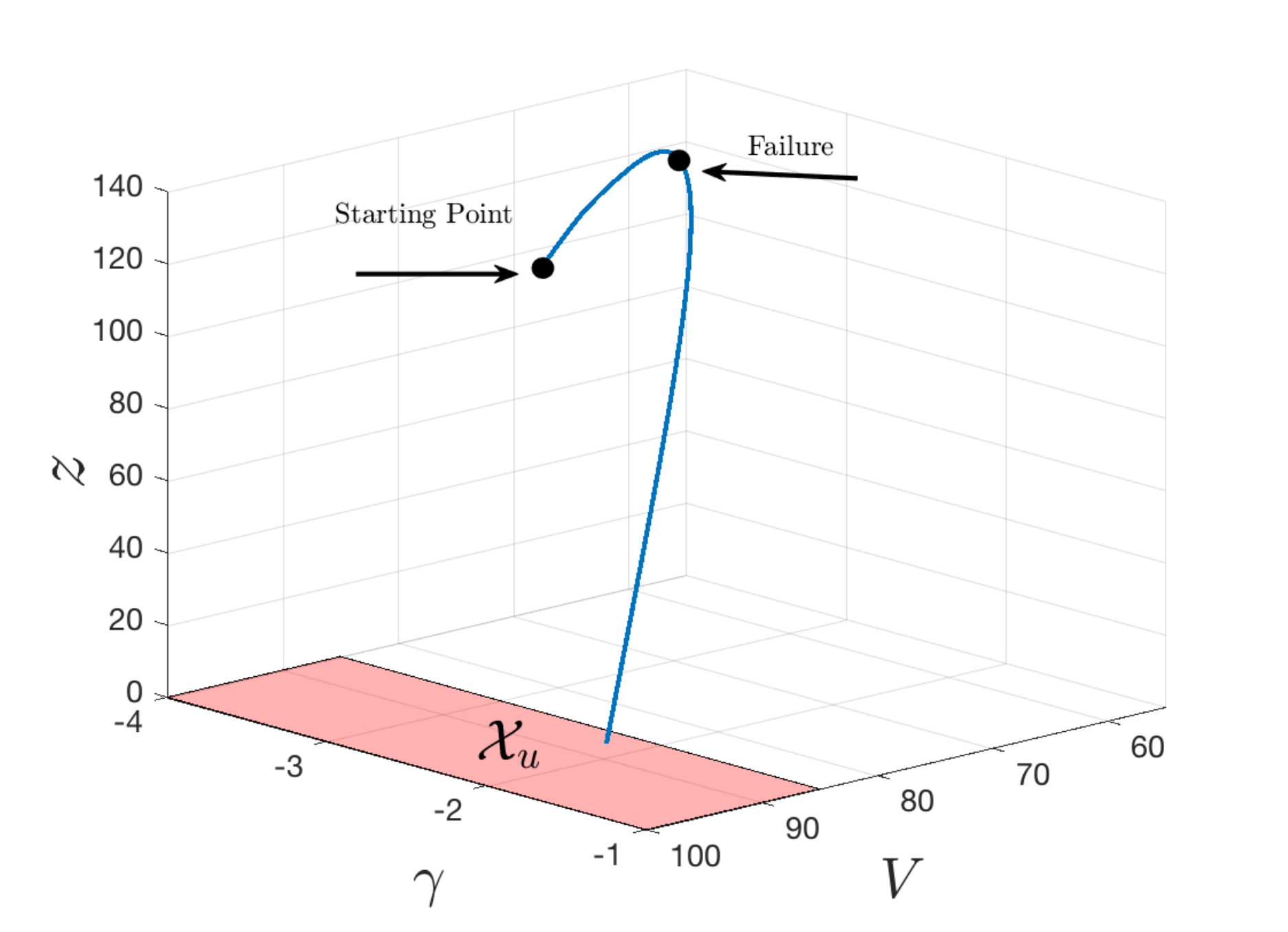}\\
  \caption{The solution of the plane system (blue) and the unsafe set (red surface). }\label{L2curve2cxcxcxc2}
\end{figure}

In the second scenario, we consider an engine control failure which is modeled as a sudden surge in thrust (to $0.8T_{max}$ from $0.2T_{max}$) at $t=0.5(s)$. 25 data samples are collected  non-uniformly until $t=2(s)$. The simulation results show that the system trajectories without the controller are not safe as shown on Figure~\ref{L2curve2cxcxcxc2}.  A data-driven differential inclusion is constructed using the side information and piecewise polynomial interpolation. We then allow $T_g=5(s)$ to calculate the safe controller using Lemma~\ref{lem5651}. Figure~\ref{L2curve2cxdddcxcxc2} shows the result of applying the safe controller obtained based on Lemma 2 with certificates $B$ and $W$ of degree $3$. Hence, the safe controller is able to ensure safe landing despite the critical failure of the aircraft.

\begin{figure}
  \centering
  \includegraphics[scale=.45]{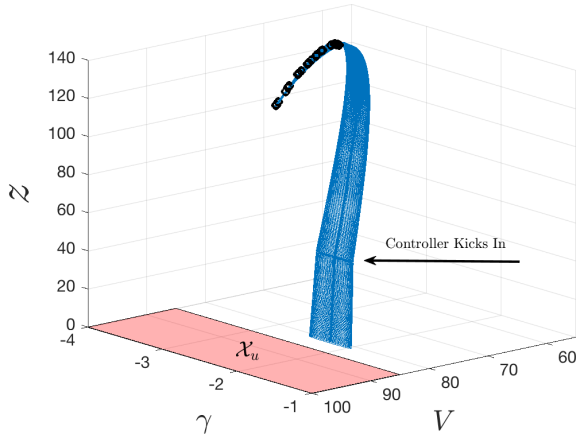}\\
  \caption{The solution of the data driven differential inclusion system (meshed surface), data points (black circles) and the unsafe set (red surface). Note that, as the safe controller kicks in, the state evolution  of the system becomes~safe.}\label{L2curve2cxdddcxcxc2}
\end{figure}

\section{CONCLUSIONS AND FUTURE WORK} \label{sec:conclusions}

\subsection{Conclusions}

We considered the problem of safety analysis  and controller synthesis for safety of systems for which only limited data and some regularity side information on system states are available. We reformulated the problem into safety analysis  and safe controller synthesis of differential inclusions.  We proposed a solution established upon an extension of barrier certificates for differential inclusions. In the case of piecewise-polynomial approximations of data, we showed that the barrier certificates can be found by polynomial optimization. Two examples were used to illustrate the proposed approach.

\subsection{Future Work}


In this study, we assumed the measurements of the states are not noisy. In many practical situations, this is not the case and sensor measurements are subject to measurement noise, say due to heat. In this setting, safety analysis requires side information in the probabilistic sense. In this respect, one can use notions such as spline smoothing~\cite{cite-keycdfd}.


The application of the proposed safety analysis results in this paper are not only limited to data-driven differential inclusions but also the discussions in Section~\ref{sec:ddss|} can be used to tackle safety analysis of discontinuous and hybrid systems, such as mechanical system with impact and Columb friction~\cite{posa2016stability}.

%


\bibliography{references}
\bibliographystyle{IEEEtran}
\begin{IEEEbiography}[{\includegraphics[width=1in,height=1.25in,clip,keepaspectratio]{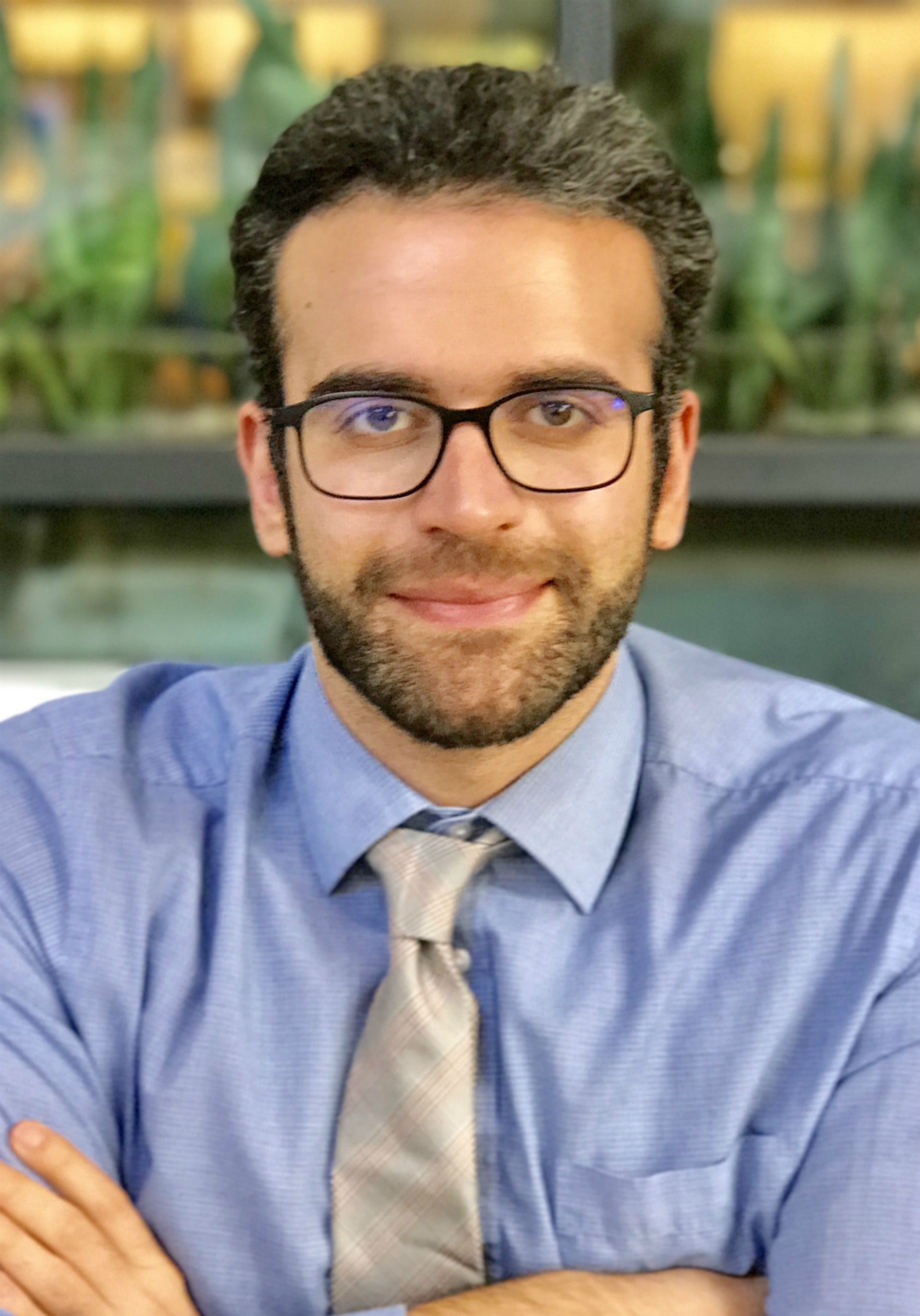}}]{Mohamadreza Ahmadi}
M. Ahmadi joined the Institute for Computational Engineering and Sciences (ICES) at the University of Texas at Austin  as a postdoctoral scholar in Fall 2016. He received his DPhil in Engineering Science (Controls) from the University of Oxford in Fall 2016 as a member of Keble College and a Clarendon Scholar. His current research interests are at the intersection of learning and control, games on Markov decision processes, and computational methods for analysis of PDEs.
\end{IEEEbiography}
\begin{IEEEbiography}[{\includegraphics[width=1in,height=1.25in,clip,keepaspectratio]{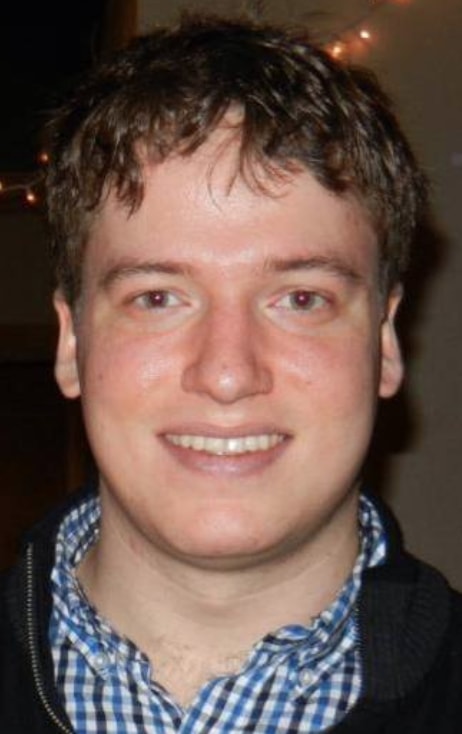}}]{Arie Israel}
Arie Israel joined the Department of Mathematics at the University of Texas at Austin as an assistant professor in Fall 2014. He received his Ph.D. degree from Princeton University in 2011, and held a postdoctoral position at New York University. His primary research focus has been on the theoretical and algorithmic foundations of extension and interpolation problems in smooth function spaces. His work draws on tools from Harmonic Analysis and PDE. Recently he has been exploring applications of his work to machine learning and control theory. 
\end{IEEEbiography}
\vspace{-18cm}
\begin{IEEEbiography}[{\includegraphics[width=1in,height=1.25in,clip,keepaspectratio]{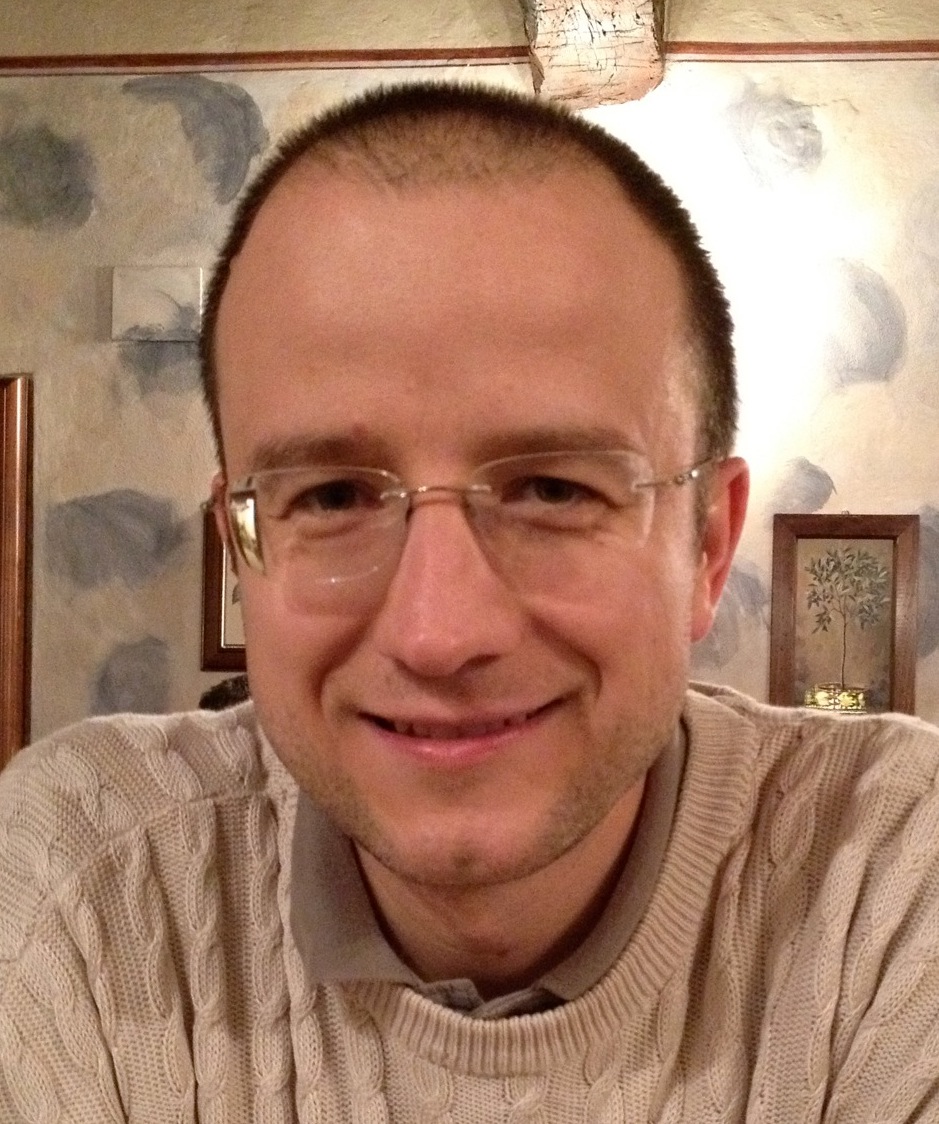}}]{Ufuk Topcu}
Ufuk Topcu joined the Department of Aerospace Engineering at the University of Texas at Austin as an assistant professor in Fall 2015. He received his Ph.D. degree from the University of California at Berkeley in 2008. He held research positions at the University of Pennsylvania and California Institute of Technology. His research focuses on the theoretical, algorithmic and computational aspects of design and verification of autonomous systems through novel connections between formal methods, learning theory and controls. 
\end{IEEEbiography}

\end{document}